\documentclass[11pt]{amsart}

\usepackage{amsmath, amssymb, amscd, mathrsfs, url, pinlabel,hyperref, verbatim}
\usepackage[margin=1.25in,marginparwidth=1in,centering,letterpaper,dvips]{geometry}
\usepackage{color,dcpic,latexsym,graphicx,epstopdf,comment}
\usepackage[all]{xy}
\usepackage{tikz}
\usepackage{enumitem,upquote}
\title{L-space knots are fibered and strongly quasipositive}

\author[John A. Baldwin]{John A. Baldwin}
\address{Department of Mathematics \\ Boston College}
\email{john.baldwin@bc.edu}

\author[Steven Sivek]{Steven Sivek}
\address{Department of Mathematics\\Imperial College London}
\email{s.sivek@imperial.ac.uk}

\def\Q{{\mathbb{Q}}}
\def\CP{{\mathbb{CP}}}
\def\cpbar{\overline{\CP}^2}

\newcommand\hf{\widehat{\mathit{HF}}}
\newcommand\hfp{\mathit{HF}^+}
\newcommand\hfred{\mathit{HF}_{\mathrm{red}}}

\DeclareMathOperator{\coker}{coker}

\newcommand\cT{\mathcal{T}}

\newcommand\Sc{\text{Spin}^c}
\newcommand\spc{\mathfrak{s}}
\newcommand\spt{\mathfrak{t}}

\newcommand\Span{\mathrm{Span}}

\newcommand\Z{\mathbb{Z}}
\newcommand\F{\mathbb{F}}

\newcommand\Img{{\rm Im}}

\newcommand\ssl{\mathit{sl}}
\newcommand\ttb{\mathit{tb}}
\newcommand\maxsl{\overline{\ssl}}

\newcommand{\longcomment}[2]{#2}

\DeclareFontFamily{U}{mathx}{\hyphenchar\font45}
\DeclareFontShape{U}{mathx}{m}{n}{
      <5> <6> <7> <8> <9> <10>
      <10.95> <12> <14.4> <17.28> <20.74> <24.88>
      mathx10
      }{}
\DeclareSymbolFont{mathx}{U}{mathx}{m}{n}
\DeclareFontSubstitution{U}{mathx}{m}{n}
\DeclareMathAccent{\widecheck}{0}{mathx}{"71}
\newcommand{\HMto}{\widecheck{\mathit{HM}}}

\longcomment{
    \RequirePackage{rotating}                   
    \def\HMto{%
       \setbox0=\hbox{$\widehat{\mathit{HM}}$}
       \setbox1=\hbox{$\mathit{HM}$}
       \dimen0=1.1\ht0
       \advance\dimen0 by 1.17\ht1
       \smash{\mskip2mu\raise\dimen0\rlap{%
          \begin{turn}{180}
              {$\widehat{\phantom{\mathit{HM}}}$}
           \end{turn}} \mskip-2mu    
                \mathit{HM}
    }{\vphantom{\widehat{\mathit{HM}}}}{}}
}
    
    \newcommand*\oline[1]{%
  \vbox{%
    \hrule height 0.35pt
    \kern0.1ex
    \hbox{%
      \kern-0.0em
      \ifmmode#1\else\ensuremath{#1}\fi
      \kern-0.1em
    }
  }
}

\newtheorem{theorem}{Theorem}
\newtheorem{lemma}[theorem]{Lemma}

\newtheorem{corollary}[theorem]{Corollary}
\newtheorem{proposition}[theorem]{Proposition}

\theoremstyle{definition}

\newtheorem{remark}[theorem]{Remark}

\makeatletter
\newtheorem*{rep@thm}{\rep@title}
\newcommand{\newreptheorem}[2]{%
\newenvironment{rep#1}[1][0,0]{%
\def\rep@title{#2##1}%
\begin{rep@thm}}%
{\end{rep@thm}}}
\makeatother
\newreptheorem{theorem}{}

\tikzset{every picture/.style=thick}

\begin{document}

\begin{abstract}
We give a new, conceptually simpler proof of the fact that knots in $S^3$ with positive L-space surgeries are fibered and strongly quasipositive.  Our motivation for doing so is that this new proof uses comparatively little Heegaard Floer-specific machinery and can thus be translated to other forms of Floer homology. We  carried this out for instanton Floer homology in our recent article ``Instantons and L-space surgeries" \cite{bs-lspace}, and used it to generalize Kronheimer and Mrowka's results on $SU(2)$ representations of fundamental groups of Dehn surgeries.
\end{abstract}

\maketitle

The hat version $\hf(Y)$ of Heegaard Floer homology, which we will take with coefficients in $\F = \Z/2\Z$ throughout, carries an absolute $\Z/2\Z$ grading such that
\begin{equation} \label{eqn:chi-hf}
\chi(\hf(Y,\spc)) = \begin{cases} 1 & b_1(Y) = 0 \\ 0 & b_1(Y) \geq 1 \end{cases}
\end{equation}
for all $\spc \in \Sc(Y)$ \cite[Proposition~5.1]{osz-properties}.  Thus for any rational homology $3$-sphere $Y$, we have
\[ \dim \hf(Y) \geq \chi(\hf(Y)) = |H_1(Y;\Z)|. \] 
A rational homology $3$-sphere $Y$ is an \emph{L-space} if
\[ \dim \hf(Y) = |H_1(Y;\Z)|. \]
 
\begin{theorem}[\cite{osz-lens,ni-hfk,hedden-positivity,osz-rational}]
\label{thm:main}
If $S^3_r(K)$ is an L-space for some rational slope $r>0$, then $K$ is fibered and strongly quasipositive, and $r \geq 2g(K)-1$.  
\end{theorem}

All proofs of Theorem \ref{thm:main} in the literature use at least some of the following tools: the doubly-filtered Heegaard Floer complex associated to a knot, the large integer surgery formula, the $(\infty,0,n)$-surgery exact triangle for $n>1$, and the $\Sc$ decomposition of $\hf(Y)$ for $Y$ a rational homology sphere.  This presents a major difficulty if one wishes to port this theorem to the instanton Floer setting, where none of this machinery is available. 

\begin{remark} A primary motivation for proving an analogue of Theorem \ref{thm:main} in the instanton Floer setting in particular is that such an analogue can be used to prove new results about the $SU(2)$ representation varieties of fundamental groups of $3$-manifolds obtained by Dehn surgeries on knots in the $3$-sphere, about which relatively little is known; see \cite{bs-lspace}.
\end{remark}

\begin{remark}
Some of the structure mentioned above is known to exist in monopole Floer homology, though not enough of it to translate previous proofs of Theorem \ref{thm:main} to that setting. The new proof of Theorem \ref{thm:main} presented in this article (see below) \emph{can} be adapted directly to monopole Floer homology, with the caveat in Remark \ref{rmk:tightknot}, to give a proof of the monopole Floer analogue of Theorem \ref{thm:main} which does not rely on an isomorphism between monopole Floer homology and Heegaard Floer homology.
\end{remark}
 
The goal of this note is to give a proof of Theorem~\ref{thm:main} using instead: the $(\infty,0,1)$-surgery exact triangle, the blow-up formula for cobordism maps, the adjunction inequality for cobordism maps,  the $\Sc$ decomposition of the  maps associated to $2$-handle cobordisms, and Ozsv{\'a}th and Szab{\'o}'s  description of the contact invariant $c^+(\xi)$ as the image of a certain class under the $2$-handle cobordism map \begin{equation*}\label{eqn:0surg}\hfp(-S^3_0(K))\to\hfp(-S^3),\end{equation*} where $K$ is a fibered knot supporting the contact structure $\xi$ on $S^3$. The first four of these tools will be used to show that an L-space knot is fibered, while the last will be used to prove that an L-space knot supports the tight contact structure on $S^3$ and is therefore strongly quasipositive, by Hedden \cite{hedden-positivity}. Strong quasipositivity will then be used to prove the $ 2g(K)-1$ bound on L-space surgery slopes.
 
\begin{remark} \label{rmk:tightknot}Ozsv{\'a}th and Szab{\'o} do not prove that $c^+(\xi)$ is well-defined (and hence that it certifies that $\xi$ is tight) directly from its description in terms of the cobordism map associated to $0$-surgery on the supporting fibered knot (and it is unclear how to do so---this is an interesting problem!). They instead use the knot filtration for this, which poses a challenge for translating the strong quasipositivity argument presented here to framed instanton homology. We discovered \cite{bs-lspace} a workaround in that setting, however, by a significantly more complicated argument which involves cabling and our framed instanton contact invariant \cite{bs-instanton}. We then used that instanton contact class to prove the $r \geq 2g(K){-}1$ bound, in a manner very similar to the proof of Proposition \ref{prop:l-space-range} here (also using results from \cite{bs-stein} and \cite{lpcs}). The same difficulties and solutions apply in monopole Floer homology, using our contact invariant from \cite{bsSHM}.\footnote{It is reasonable to expect that Kronheimer and Mrowka's monopole Floer contact class can be characterized in terms of the $0$-surgery cobordism map as above, based on Echeverria's work \cite{echeverria}, which would allow one to circumvent the more complicated strong quasipositivity argument we have in mind.}
\end{remark}

In our proof of Theorem~\ref{thm:main}, we will assume that $K$ has genus at least 2 everywhere until Proposition~\ref{prop:l-space-range}, so that we can apply Theorem~\ref{thm:ni-fibered} below to detect whether $S^3_0(K)$ is fibered.  The following proposition uses a cabling trick (see Figure~\ref{fig:cabling}) to show that we can still conclude Theorem~\ref{thm:main} in full generality, and moreover that it suffices to prove the theorem for integral slopes.

\begin{proposition} \label{prop:genus-1-reduction}
If Theorem~\ref{thm:main} holds for all knots of genus at least 2 and integral slopes $r \in \Z$, then it is also true for knots of genus 1 and $r \in \Q$.
\end{proposition}

\begin{proof}
The claim in Theorem~\ref{thm:main} about the set of positive integral L-space slopes is proved in Proposition~\ref{prop:l-space-range} without any restrictions on $g(K)$, so we will only address the other claims of Theorem~\ref{thm:main} here.

We suppose first that $K$ is an arbitrary nontrivial knot and that some surgery on $K$ of non-integral slope $r > 0$ is an L-space.  We write $r = \frac{p}{q}$ for some positive integers $p$ and $q \geq 2$.  By applying \cite[Corollary~7.3]{gordon}, we see that
\[ S^3_{pq}(K_{p,q}) \cong S^3_{p/q}(K) \# S^3_{q/p}(U) \]
where $K_{p,q}$ is the cable represented by the peripheral element $\mu^p\lambda^q$ in $\pi_1(\partial N(K))$.  The two summands on the right are both L-spaces, hence the K\"unneth formula for $\hf$ says that $pq$-surgery on $K_{p,q}$ is also an L-space.  We observe that
\[ g(K_{p,q}) = \frac{(p-1)(q-1)}{2} + q\cdot g(K), \]
which implies that $g(K_{p,q}) \geq q \geq 2$ and which is also equivalent to
\begin{equation} \label{eq:cable-genus-bound}
2g(K_{p,q}) - 1 = pq + q\left(2g(K)-1 - \frac{p}{q}\right).
\end{equation}

We can now apply the assumed case of Theorem~\ref{thm:main} to $pq$-surgery on $K_{p,q}$ to conclude that $pq \geq 2g(K_{p,q})-1$, hence $r = \frac{p}{q} \geq 2g(K)-1$ by \eqref{eq:cable-genus-bound}; and that $K_{p,q}$ is fibered and strongly quasipositive.  Since $K_{p,q}$ is fibered, $K$ must be as well.  The strong quasipositivity of $K$ then follows from two facts: (1) a fibered knot is strongly quasipositive if and only if its corresponding open book decomposition supports the tight contact structure on $S^3$ \cite{hedden-positivity}, and (2) the knots $K$ and $K_{p,q}$ support the same contact structure \cite{hedden-cabling}.

This concludes the proof in all cases except when $K$ has genus 1 and $r$ is a positive integer, say $r=n$.  In this case it is automatic that $r \geq 2g(K)-1 = 1$.  Moreover, repeated application of \cite[Proposition~2.1]{osz-lens}, which follows easily from the surgery exact triangle for $\hf$, says that $S^3_{(2n+1)/2}(K)$ is an L-space.  (In fact, it says that $S^3_s(K)$ is an L-space for all rational $s \geq n$.)  Since $\frac{2n+1}{2} \not\in \Z$, the fiberedness and strong quasipositivity of $K$ follow exactly as above.
\begin{figure}[t]
\includegraphics[scale=0.6]{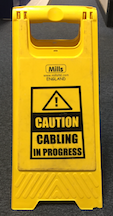}
\caption{An outline of the proof of Proposition~\ref{prop:genus-1-reduction}.} \label{fig:cabling}
\end{figure}
\end{proof}
 
We will suppose henceforth that $K\subset S^3$ is a knot of genus $g\geq 2$.

Let $\Sigma_0$ denote the genus $g$ surface in $S^3_0(K)$ obtained by capping off a minimal genus Seifert surface for $K$. Let $\spc_i$ be the unique $\Sc$ structure on $S^3_0(K)$ satisfying
\[\langle c_1(\spc_i),[\Sigma_0]\rangle = 2i.\]
The adjunction inequality \cite[Theorem~7.1]{osz-properties} implies that
\[\hfp(S^3_0(K),\spc_i)=\hf(S^3_0(K),\spc_i)=0\]
for $|i|>g-1$. Moreover,
\[\hfp(S^3_0(K),\spc_{g-1})\neq 0 \]
\cite{osz-genus}, and Ni proved \cite{ni-hfk} (see also \cite[Corollary~4.5]{osz-hfk}) the following.

\begin{theorem} \label{thm:ni-fibered}
$K$ is fibered if and only if $\hfp(S^3_0(K),\spc_{g-1})\cong \F$.
\end{theorem} 

Recall that there is an exact triangle 
\begin{equation} \label{eqn:trianglehf}
\begin{gathered}
\xymatrix@C=0pt@R=35pt{
\hfp(Y,\spc) \ar[rr]^-{\cdot \,U} & & \hfp(Y,\spc) \ar[dl]^-{j} \\
& \hf(Y,\spc), \ar[ul]^-{i}  \\
} 
\end{gathered}
\end{equation}
where $i$ and multiplication by $U$ preserve the  $\Z/2\Z$ grading and $j$ shifts it by $1$.
Moreover, we claim the following.

\begin{proposition}
\label{prop:Utrivial}
$U$ acts trivially on $\hfp(S^3_0(K),\spc_{g-1})$.
\end{proposition} 

\begin{proof}
Let
\[ Z:(\Sigma_0\times S^1)\sqcup S^3_0(K)\to S^3_0(K) \]
be the cobordism obtained from $S^3_0(K)\times I$ by removing a neighborhood of $\Sigma_0$ from the interior. Since
\begin{equation}\label{eqn:prod}
\hfp(\Sigma_0\times S^1,\spc_{g-1})\cong \F,
\end{equation}
the induced map
\[ F_Z:\hfp(\Sigma_0\times S^1,\spc_{g-1})\otimes \hfp(S^3_0(K),\spc_{g-1})\to \hfp(S^3_0(K),\spc_{g-1}) \]
is an isomorphism. Moreover, it satisfies
\[ F_Z(a\otimes Ub) = F_Z(Ua\otimes b). \]
The $U$-action on \eqref{eqn:prod} is clearly trivial, which then implies the same for $\hfp(S^3_0(K),\spc_{g-1})$ by the relation above.
\end{proof}

Proposition \ref{prop:Utrivial} together with the exact triangle in \eqref{eqn:trianglehf} implies that 
\[ \hf(S^3_0(K),\spc_{g-1}) \cong \hfp(S^3_0(K),\spc_{g-1})\oplus \hfp(S^3_0(K),\spc_{g-1})[1].\]
In particular, we have the following.

\begin{corollary} \label{cor:hatfibered}
If $K$ is fibered then
\[ \hf(S^3_0(K),\spc_{g-1}) \cong \F_0 \oplus \F_1, \]
where the subscripts on the right denote the $\Z/2\Z$ grading.  If $K$ is not fibered then
\[ \dim \hf(S^3_0(K),\spc_{g-1}) \geq 4. \]
\end{corollary}

We now consider the natural 2-handle cobordisms 
\[ S^3 \xrightarrow{X_k} S^3_k(K) \xrightarrow{W_{k+1}} S^3_{k+1}(K) \]
for each integer $k \geq 0$, where the 2-handle $W_{k+1}$ is attached along a $-1$-framed meridian of $K$.  We observe in Figure~\ref{fig:handleslide} that
\[ W_{k+1} \circ X_k = X_k \cup_{S^3_k(K)} W_{k+1} \cong X_{k+1} \# \cpbar, \]
\begin{figure}[t]
\begin{tikzpicture}
\draw (-1,1) node[right] {$k$} -- (-1,-0.3) (-1,-0.5) -- (-1,-1) node[below] {$K$};
\path (-1,-0.4) arc (-90:75:0.5 and 0.25) node (m) {} arc (75:0:0.5 and 0.25) node[right] {$-1$};
\draw (m) arc(75:-255:0.5 and 0.25);
\node at (1,-0.15) {$\cong$};
\draw (2,1) node[right]{$k+1$} -- (2,-1) node[below] {$K$};
\draw (3,-0.15) ellipse (0.5 and 0.25);
\node[right] at (3.5,-0.15) {$-1$};
\end{tikzpicture}
\caption{A handleslide showing that $W_{k+1} \circ X_k \cong X_{k+1} \# \cpbar$.}
\label{fig:handleslide}
\end{figure}
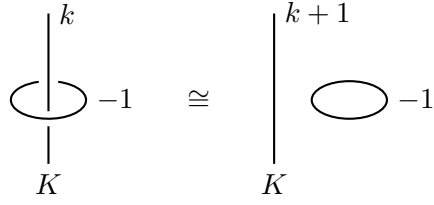%
and hence if we write 
\[ V_k = W_k \circ W_{k-1} \circ \dots \circ W_1: S^3_0(K) \to S^3_k(K) \]
then the composition
\[ Z_k = V_k \circ X_0: S^3 \to S^3_k(K) \]
is a $k$-fold blow-up of $X_k$, i.e.,
\begin{align*}
X_k \# k\cpbar &\cong W_k \circ (X_{k-1} \# (k-1)\cpbar) \\
&\cong W_k \circ W_{k-1} \circ (X_{k-2} \# (k-2)\cpbar) \\
&\cong \dots \\
&\cong (W_k \circ \dots \circ W_1) \circ X_0 = Z_k.
\end{align*}
The maps induced by $X_k$ and $W_{k+1}$ fit into an $(\infty,0,1)$-surgery exact triangle,
\begin{equation} \label{eqn:triangle}
\begin{gathered}
\xymatrix@C=-15pt@R=30pt{
\hf(S^3) \ar[rr]^{F_{X_k}} & & \hf(S^3_k(K)) \ar[dl]^{F_{W_{k+1}}} \\
& \hf(S^3_{k+1}(K)). \ar[ul]  \\
}
\end{gathered}
\end{equation}

A $\Sc$ structure on $X_0$ is determined by its restriction to $S^3_0(K)$, or, equivalently, by  the evaluation of its first Chern class on $[\Sigma_0]$. Let $\spt_i$ denote the unique $\Sc$ structure on $X_0$ with
\[\langle c_1(\spt_i),[\Sigma_0]\rangle=2i.\]
Define
\[y_i := F_{X_0,\spt_i}(\mathbf{1})\in \hf(S^3_0(K),\spc_i),\]
where $\mathbf{1}$ denotes the generator of $\hf(S^3) \cong \F$.

Let $\Sigma_k$ denote the capped off Seifert surface in $X_k$, with
\[\Sigma_k\cdot\Sigma_k=k.\]
A $\Sc$ structure on $X_k$ is determined by the evaluation of its first Chern class on $[\Sigma_k]$. Such Chern classes are characteristic elements, so this evaluation agrees with $k \pmod{2}$. Let $\spt_{k,i}$ denote the unique $\Sc$ structure on $X_k$ satisfying
\[\langle c_1(\spt_{k,i}),[\Sigma_k]\rangle+k=2i.\]
The adjunction inequality \cite[Proof of Theorem~1.5]{osz-cobordism} implies that the map
\[F_{X_k,\spt_{k,i}}:\hf(S^3)\to\hf(S^3_k(K))\]
is nontrivial only if
\[|\langle c_1(\spt_{k,i}),[\Sigma_k]\rangle|+k\leq 2g-2, \]
or equivalently $1-g+k\leq i \leq g-1$.

\begin{lemma} \label{lem:compute-F_V}
Let $x_{k,i} = F_{X_k,\spt_{k,i}}({\mathbf 1})$ for all $k \geq 1$ and all $i$.  Then
\[ F_{V_k}(y_i) = x_{k,i} + \binom{k}{1} x_{k,i+1} + \binom{k}{2} x_{k,i+2} + \dots + \binom{k}{g-i-1} x_{k,g-1} \]
as elements of $\hf(S^3_k(K))$.
\end{lemma}

\begin{proof}
Let $E_1,\dots, E_k \subset Z_k$ denote the exceptional spheres in $Z_k \cong X_k \# k\cpbar$, and $e_1,\dots,e_k$  their Poincar{\'e} duals in $H^2(Z_k)$.
Note that in $Z_k$, the surface $\Sigma_0$ is given by
\[\Sigma_0 = \Sigma_k-E_1-\dots -E_k.\]
In particular,
\begin{equation}\label{eqn:eval-zk}
\langle c_1(\spt_{k,i}+a_1e_1+\dots + a_ke_k),[\Sigma_0]\rangle=2i-k+a_1+\dots+a_k
\end{equation}
in $Z_k$.  We will evaluate $F_{Z_k}$ by applying the blow-up formula for cobordism maps \cite[Theorem~3.7]{osz-cobordism}, which says that for a $\Sc$ cobordism
\[(W,\spt):(Y_1,\spc_1)\to(Y_2,\spc_2)\]
with blow-up $\hat{W} = W\#\overline{\CP}^2$ and exceptional sphere $E$, 
\[F_{\hat{W},\spt\pm(2\ell+1)\mathit{PD}(E)} = \begin{cases} F_{W,\spt} & \ell=0 \\ 0 &\ell\neq 0 \end{cases}\] 
as maps on $\hf$ for any $\ell \geq 0$.

Let $F_i$ denote the component of $F_{Z_k} = F_{V_k} \circ F_{X_0}$ that factors through $\hf(S^3_0(K),\spc_i)$.  On the one hand, we have
\[ F_i = F_{V_k} \circ F_{X_0,\spt_i}: \hf(S^3) \to \hf(S^3_k(K)). \]
 On the other hand, if we let $e = e_1+\dots+e_k$, then for each $i$ we have
 \begin{align*}
 F_{i}&=F_{Z_k,\spt_{k,i}+e} \\
 &\quad+ \sum_{j_1} F_{Z_k,\spt_{k,i+1}+e-2e_{j_1}}\\
 &\quad+ \sum_{j_1<j_2} F_{Z_k,\spt_{k,i+2}+e-2e_{j_1}-2e_{j_2}} \\
 &\quad+\dots \\
 &\quad+\sum_{j_1<\dots<j_{g-i-1}} F_{Z_k,\spt_{k,g-1}+e-2e_{j_1}+\dots-2e_{j_{g-i-1}}},
\end{align*}
by the formula \eqref{eqn:eval-zk}.
From the blow-up formula, we have
\[ F_{Z_k,\spt_{k,j}\pm e_1\pm\dots \pm e_k} = F_{X_k,\spt_{k,j}}, \]
so the expression for $F_{i}$ above becomes
\[ F_{i}=F_{X_k,\spt_{k,i}} 
 + \binom{k}{1} F_{X_k,\spt_{k,i+1}}
 + \binom{k}{2} F_{X_k,\spt_{k,i+2}} 
 +\dots 
 +\binom{k}{g-i-1} F_{X_k,\spt_{k,g-1}}. \]
We conclude by evaluating both sides on the element ${\mathbf 1} \in \hf(S^3)$.
\end{proof}

\begin{proposition} \label{prop:key-kernel}
For all integers $k \geq 1$, we have
\[ \ker \big(F_{V_k}: \hf(S^3_0(K)) \to \hf(S^3_k(K))\big) \subset \Span_\F(y_{1-g},\dots,y_{g-1}). \]
This inclusion is an equality for all $k \geq 2g-1$.
\end{proposition}

\begin{proof}
When $k=1$, the exact triangle \eqref{eqn:triangle} says that
\[ \ker(F_{V_1}) = \ker(F_{W_1}) = \Img(F_{X_0}) = \Span_\F\left(\sum_{i=1-g}^{g-1} y_i\right). \]
We prove the inclusion in general by induction on $k$.

Suppose that $k \geq 1$, and fix an element $z \in \ker(F_{V_{k+1}})$.  Then
\[ F_{W_{k+1}}\big(F_{V_k}(z)\big) = 0 \]
by definition, so the exact triangle \eqref{eqn:triangle} tells us that $F_{V_k}(z) \in \Img(F_{X_k})$, or equivalently
\begin{equation} \label{eq:V_k-x}
F_{V_k}(z) = c\cdot F_{X_k}({\mathbf 1})
\end{equation}
for some $c\in\F$.  Lemma~\ref{lem:compute-F_V} says that each element
\[ x_{k,i} = F_{X_k,\spt_{k,i}}({\mathbf 1}) \in \hf(S^3_k(K)) \]
is a linear combination of the various $F_{V_k}(y_i)$, since the matrix of the coefficients of the system of linear equations relating $(F_{V_k}(y_i))_i$ to $(x_{k,i})_i$ is triangular and clearly invertible.  In particular, summing over all $i$ reveals that
\begin{equation} \label{eq:X_k-1-span}
F_{X_k}({\mathbf 1}) \in \Span_\F\big(F_{V_k}(y_{1-g}),\dots,F_{V_k}(y_{g-1})\big).
\end{equation}
Combining \eqref{eq:V_k-x} and \eqref{eq:X_k-1-span}, there are coefficients $a_j \in \F$ such that
\[ F_{V_k}(z) = c \cdot \sum_{j=1-g}^{g-1} a_j F_{V_k}(y_j), \]
or equivalently
\begin{equation} \label{eq:z-yj-combo}
z - \sum_{j=1-g}^{g-1} ca_j\cdot y_j \in \ker F_{V_k}.
\end{equation}
By induction the left side of \eqref{eq:z-yj-combo} lies in $\Span_\F(y_{1-g},\dots,y_{g-1})$, hence the same is true of $z$.  Since $z$ was an arbitrary element of $\ker F_{V_{k+1}}$, this completes the inductive step.

To see that equality holds when $k \geq 2g-1$, we observe from Lemma~\ref{lem:compute-F_V} that $F_{V_k}(y_i)$ is a linear combination of various elements $x_{k,j} = F_{X_k,\spt_{k,j}}({\mathbf 1})$.  Using the adjunction inequality, we have already noted that $x_{k,j} = 0$ unless
\[ 1-g+k \leq j \leq g-1, \]
so for $k \geq 2g-1$ the elements $x_{k,j}$ and hence the $F_{V_k}(y_i)$ are all zero.
\end{proof}

\begin{proposition} \label{prop:hf-zero-surgery}
Suppose that $S^3_n(K)$ is an L-space for some positive integer $n$.  Then
\[ \hf(S^3_0(K),\spc_j) = \begin{cases} \F_0 \oplus \F_1 & y_j \neq 0 \\ 0 & y_j = 0 \end{cases} \]
for all $j$, where the subscripts on each copy of $\F$ denote the $\Z/2\Z$ grading.
\end{proposition}

\begin{proof}
We observe from \eqref{eqn:triangle} that
\begin{equation} \label{eq:rank-change}
\dim_\F \hf(S^3_{k+1}(K)) = \dim_\F \hf(S^3_k(K)) + \begin{cases} 1 & F_{X_k} = 0 \\ -1 & F_{X_k} \neq 0 \end{cases}
\end{equation}
for all $k \geq 0$.  If $m$ denotes the number of $k \in \{0,1,\dots,n-1\}$ such that $F_{X_k} \neq 0$, then 
\[ n = \dim_\F \hf(S^3_n(K)) = \dim_\F \hf(S^3_0(K)) + (n-m) - m, \]
which simplifies to
\begin{equation} \label{eq:s30-2m}
\dim_\F \hf(S^3_0(K)) = 2m.
\end{equation}
Our goal is thus to compute $m$.

Supposing that $F_{X_k} \neq 0$ for some $k \geq 0$, then $F_{X_k}({\mathbf 1})$ is a nonzero element which spans $\ker(W_{k+1})$, and from \eqref{eq:X_k-1-span} it has the form
\[ F_{X_k}({\mathbf 1}) = F_{V_k}\left(\sum_{j=1-g}^{g-1} a_j y_j \right) \]
for some coefficients $a_j \in \F$.  The sum $\sum a_j y_j$ is thus not in $\ker (F_{V_k})$, but it is in
\[ \ker(F_{V_{k+1}}) = \ker(F_{W_{k+1}} \circ F_{V_k}), \]
so we have $\dim \ker(F_{V_{k+1}}) > \dim \ker(F_{V_k})$.  This  implies that \[\dim\ker(F_{V_{n}})\geq m.\] Proposition~\ref{prop:key-kernel} then implies that
\begin{equation} \label{eq:m-span-yj}
m \leq \dim \Span_\F(y_j). 
\end{equation} 
But the nonzero $y_j$ are all linearly independent, since they belong to different summands $\hf(S^3_0(K),\spc_j)$ of $\hf(S^3_0(K))$, so by combining \eqref{eq:s30-2m} and \eqref{eq:m-span-yj} we conclude that
\begin{equation} \label{eq:dim-s30-bound}
\dim_\F \hf(S^3_0(K)) \leq 2\cdot \#\{j \mid y_j \neq 0\}.
\end{equation}
If $y_j \neq 0$ then $\hf(S^3_0(K),\spc_j)$ is nonzero, and its Euler characteristic is zero by \eqref{eqn:chi-hf}, so
\[ \F_0 \oplus \F_1 \subset \hf(S^3_0(K),\spc_j) \mathrm{\ if\ } y_j \neq 0. \]
Thus the inequality in \eqref{eq:dim-s30-bound} must be an equality, and each nonzero $\hf(S^3_0(K),\spc_j)$ must have the form $\F_0 \oplus \F_1$, completing the proof.
\end{proof}


\begin{proposition} \label{prop:fibered}
If $S^3_n(K)$ is an L-space for some integer $n>0$ then $K$ is fibered.
\end{proposition}

\begin{proof}
Corollary~\ref{cor:hatfibered} and Proposition~\ref{prop:hf-zero-surgery} tell us that
\begin{equation} \label{eq:top-spc-dim}
2 \leq \dim \hf(S^3_0(K),\spc_{g-1}) \leq 2,
\end{equation}
and that equality on the left holds if and only if $K$ is fibered, so $K$ must be fibered.
\end{proof}

\begin{proposition}
If $S^3_n(K)$ is an L-space for some integer $n>0$ then $K$ is strongly quasipositive.
\end{proposition} 
\begin{proof}
We already have seen in \eqref{eq:top-spc-dim} that $\hf(S^3_0(K),\spc_{g-1})$ is nonzero, hence $y_{g-1} \neq 0$ by Proposition~\ref{prop:hf-zero-surgery}.
Equivalently, the map
\begin{equation}\label{eqn:x0map}
\hf(S^3)\to\hf(S^3_0(K),\spc_{g-1})
\end{equation}
induced by $X_0$ is nonzero. Now, we can also view $X_0$ as a cobordism
\[ X_0:-S^3_0(K)\to-S^3, \]
in which case the induced map
\[ \hf(-S^3_0(K),\spc_{1-g})\to\hf(-S^3) \]
is dual to that in \eqref{eqn:x0map}. In particular, this map is also nonzero. The commutativity of 
\begin{equation*} \xymatrix@C=20pt@R=30pt{
\hf(-S^3_0(K),\spc_{1-g}) \ar[r]^-{i} \ar[d]&\hfp(-S^3_0(K),\spc_{1-g})\ar[d]  \\
\hf(-S^3) \ar[r]^-{i} &\hfp(-S^3), \\
} \end{equation*}
where the vertical maps are those induced by $X_0$, together with the facts that $\hf(-S^3)\cong \F$ and  the bottom horizontal map is nonzero, implies that the rightmost vertical map
\begin{equation}\label{eqn:x0map2}
\hfp(-S^3_0(K),\spc_{1-g})\to\hfp(-S^3)
\end{equation}
is nonzero as well. But Proposition~\ref{prop:fibered} says that $K$ is fibered, hence
\[\hfp(-S^3_0(K),\spc_{1-g})\cong \F,\]
and the image of its generator under the map in \eqref{eqn:x0map2} is the contact invariant $c^+(\xi_K)$ \cite{osz-contact}, where $\xi_K$ is the contact structure corresponding to $K$. Thus, $c^+(\xi_K)$ is nonzero, which implies that $\xi_K$ is the tight contact structure on $S^3$. It follows that $K$ is strongly quasipositive, by work of Hedden \cite[Proposition~2.1]{hedden-positivity}.
\end{proof}

We will now use the fact that L-space knots are strongly quasipositive to determine the range of L-space slopes for any such knot.  We begin with the following general lemma.

\begin{lemma} \label{lem:dim-from-ker-U}
Let $Y$ be a rational homology sphere with $|H_1(Y;\Z)| = n$.  Suppose that
\[ \ker\left(U: \hfp(Y) \to \hfp(Y)\right) \]
has dimension $n+k$.  Then $\dim \hf(Y) = n+2k$.
\end{lemma}

\begin{proof}
The exact triangle \eqref{eqn:trianglehf} involving the $U$-action on $\hfp(Y)$ produces a short exact sequence
\[ 0 \to \coker(U) \to \hf(Y) \to \ker(U) \to 0. \]
Thus it will suffice to show that $\dim \coker(U) = k$.

Since each $\Sc$ structure on $Y$ is torsion, we have a short exact sequence
\[ 0 \to (\cT^+)^{\oplus n} \to \hfp(Y) \xrightarrow{\pi} \hfred(Y) \to 0, \]
of $\F[U]$-modules, where $\cT^+ \cong \F[U,U^{-1}] / UF[U]$.  The quotient $\hfred(Y)$ is defined as $\hfp(Y)/\Img(U^d)$ for $d \gg 0$; it is finitely generated over $\F[U]$ and over $\F$, and every element is $U$-torsion, so it has a decomposition
\[ \hfred(Y) \cong \bigoplus_{i=1}^r \F[U] / \langle U^{n_i} \rangle, \]
with each $n_i \geq 1$.  Moreover this sequence can be shown to split, so that
\[ \hfp(Y) \cong (\cT^+)^{\oplus n} \oplus \bigoplus_{i=1}^r \F[U]/\langle U^{n_i}\rangle. \]
But then it is clear that $\ker(U) \cong \F^{n+r}$, so that $r=k$, and then that $\coker(U) \cong \F^r = \F^k$, and the lemma follows immediately.
\end{proof}

The following proposition completes our proof of Theorem~\ref{thm:main}.  The proof below is partly inspired by the work of Lidman, Pinz\'on-Caicedo, and Scaduto in \cite{lpcs}.
\begin{proposition} \label{prop:l-space-range}
If $K$ has genus $g \geq 1$ and $S^3_n(K)$ is an L-space for some positive integer $n$, then $S^3_n(K)$ is an L-space for an arbitrary integer $n$ if and only if $n \geq 2g-1$.
\end{proposition}

\begin{proof}
Since $K$ is strongly quasipositive, its maximal self-linking number is $\maxsl(K) = 2g-1$.  We take a Legendrian representative $\Lambda$ of $K$ in the standard contact $S^3$ with classical invariants
\[ (\ttb(\Lambda), r(\Lambda)) = (\tau_0,r_0), \qquad \tau_0 - r_0 = 2g-1, \]
and for $n \geq 1-\tau_0$, we can positively stabilize this $k$ times and negatively stabilize it $\tau_0+n-1-k$ times to get a Legendrian representative with
\[ (\ttb,r) = (1-n, 2-2g-n+2k), \qquad 0 \leq k \leq \tau_0+n-1. \]
For odd $n \gg 0$, these values of $r$ include every positive odd number between $1$ and $n+2g-2$.

Fixing such a large value of $n$, we perform Legendrian surgery on these knots $\Lambda_i$ with
\[ (\ttb(\Lambda_i), r(\Lambda_i)) = (1-n, 2i-1), \qquad 1 \leq i \leq \frac{n+2g-1}{2}, \]
to get contact structures
\[ \xi_1,\dots,\xi_{(n+2g-1)/2} \]
on $S^3_{-n}(K)$.  If $X_{-n}(K)$ is the trace of this $-n$-surgery, and $\hat\Sigma \subset X_{-n}(K)$ the union of a Seifert surface for $K$ with the core of the 2-handle, then each $\xi_i$ admits a Stein filling $(X_{-n}(K),J_i)$ with $\langle c_1(J_i),[\hat\Sigma]\rangle = r(\Lambda_i) = 2i-1$.  We can also take contact structures
\[ \bar{\xi}_i = T(S^3_{-n}(K)) \cap \bar{J}_i T(S^3_{-n}(K)), \qquad 1 \leq i \leq \frac{n+2g-1}{2}, \]
which are filled by $X_{-n}(K)$ with the conjugate Stein structure $\bar{J}_i$ for each $i$.  These satisfy $\langle c_1(\bar{J}_i), [\hat\Sigma]\rangle = -(2i-1)$, so we have exhibited $n+2g-1$ Stein structures
\[ J_1,J_2,\dots,J_{(n+2g-1)/2}, \bar{J}_1,\bar{J}_2, \bar{J}_{(n+2g-1)/2} \]
on $X_{-n}(K)$ which are all distinguished by their first Chern classes.

A theorem of Plamenevskaya \cite{plamenevskaya} now tells us that the corresponding contact invariants
\[ c^+(\xi_1), \dots, c^+(\xi_{(n+2g-1)/2}), c^+(\bar{\xi}_1), \dots, c^+(\bar{\xi}_{(n+2g-1)/2}) \in \hfp(-S^3_{-n}(K)) \]
are linearly independent.  These elements lie in $\ker(U)$, as can be seen, for example, from the fact that they are by defined by maps of the form \eqref{eqn:x0map2} whose domains have trivial $U$ action.  Thus
\[ \dim \ker(U) \geq n+2g-1, \]
and it follows from Lemma~\ref{lem:dim-from-ker-U} that
\[ \dim \hf(S^3_{-n}(Y)) = \dim \hf(-S^3_{-n}(Y)) \geq n+4g-2. \]
This same argument applies for any larger odd value of $n$ as well, and the conclusion also holds for even values of $n$ after making only cosmetic changes to the argument, so that $S^3_{-m}(Y)$ cannot be an L-space for any $m \geq n$.

We now repeatedly apply the surgery exact triangle \eqref{eqn:triangle} to see that
\[ \dim \hf(S^3_{-m}(Y)) \geq m + 4g-2, \qquad 0 \leq m \leq n, \]
and then that
\[ \dim \hf(S^3_m(Y)) \geq 4g-2-m \geq m+2, \qquad 0 \leq m \leq 2g-2. \]
Thus $S^3_m(Y)$ cannot be an L-space for any integer $m < 2g(K)-1$.  On the other hand, equation \eqref{eq:rank-change} says that
\[ \dim \hf(S^3_{2g-1+n}(K)) = \dim\hf(S^3_{2g-1}(K)) + n \]
for all $n \geq 0$, since the maps $F_{X_{2g-1}},\dots,F_{X_{2g-2+n}}$ are all zero by the adjunction inequality.  Thus $S^3_{2g-1+n}(K)$ is an L-space if and only if $S^3_{2g-1}(K)$ is, and this completes the proof.
\end{proof}

\noindent {\bf Acknowledgments.} We thank Jen Hom and Tye Lidman for helpful conversations.  JAB was supported by NSF CAREER Grant DMS-1454865.

\bibliographystyle{alpha}
\bibliography{References}

\end{document}